\documentclass[10pt]{article}
\usepackage[english]{babel}
\usepackage[fleqn]{amsmath}
\usepackage{amssymb}
\usepackage{amsthm}
\usepackage{psfrag}
\usepackage{graphicx}
\usepackage[margin=10pt,font=footnotesize,labelfont=bf]{caption}
\usepackage[utf8]{inputenc}
\usepackage[T1]{fontenc}
\usepackage{lmodern}
\usepackage{hyperref}
\title{The logic of sheaves, sheaf forcing  and the independence of the Continuum Hypothesis}

\author{J. Benavides\footnote{Department of mathematics "Ulisse Dini",  University of Florence, navarro@math.unifi.it}}

\date{}

\newtheorem{defin}{Definition}[section]
\newtheorem{teor}{Theorem}[section]
\newtheorem{corol}{Corollary}[section]
\newtheorem{lema}{Lemma}[section]

\begin{document}

\maketitle

\begin{abstract}
An introduction is given to the logic of sheaves of structures and to set theoretic forcing  constructions based on this logic. Using these tools, it is presented an alternative proof of  the independence of the Continuum Hypothesis; which simplifies and unifies the  classical boolean and intuitionistic approaches, avoiding the difficulties linked to the categorical machinery of the topoi based approach.
\end{abstract}

\section*{Introduction}

The first of the 10 problems that Hilbert posed in his famous conference the 1900 in Paris \footnote{The Hilberts Problems are 23 but in the conference he presented just 10 problems,  problems 1, 2, 6, 7, 8, 13, 16, 19, 21 and 22 in the publish paper \cite{hilbert1}.},  the problem of the Continuum Hypothesis (CH for short), which asserts that every subset of the real numbers is either enumerable or has the cardinality of the Continuum; was the origin of probably one of the greatest achievements in Mathematics of the second half of the XX century: The forcing technique created by Paul Cohen during the sixties. Using this technique Cohen showed that  the CH is independent (i.e. it cannot be proved or disproved) from the axioms of Zermelo-Fraenkel plus the Axiom of Choice ($ZFC$).\\

At the beginning Cohen approached the problem in a pure syntactic way, he defined a procedure to find a contradiction in $ZFC$ given a contradiction in $ZFC+\neg CH$; which implies that $Con(ZFC)\rightarrow Con (ZFC+\neg CH)$ (The consistency of $ZFC$ implies the consistency of $ZFC+\neg CH$). However this approach was a non-constructivist one, for it was not given a model of $ZFC+\neg CH$. Nevertheless in his publish papers Cohen developed two other approaches more constructivist. The first one assumes the existence of a countable transitive model (c.t.m.) of $ZFC$ \footnote{Remember that by the Gödel theorem the existence of such model cannot be proved within $ZFC$.}, then considering any finite list of axioms $\varphi_1,...,\varphi_n$ of $ZFC+\neg CH$   he proved within ZFC the existence of a c.t.m  for $\varphi_1,...,\varphi_n$. The second one  formalizes logic within $ZFC$, being a formal theorem of $ZFC$ the result which expresses that  for any c.t.m. of $ZFC$ exists a c.t.m. of $ZFC+\neg CH$ which is an extension of the ground model. A couple of years later, in 1965, Scott and Solovay gave an even more constructive version of Cohen's technique  based on Boolean Valued models, i.e. models rule by a Boolean-valued logic; the advantage of this approach is  the fact that it made more evident  how a model for $ZFC+\neg CH$ is constructed. An analogous approach on Kripke Models was developed by Fitting on 1969,  he reinterpreted the forcing technique as the construction of a Kripke Model over certain partial orders in such a way that each of the fibers are models of Set Theory.  Independently around 1963, Lawvere and Tierney attempts to axiomatise the category of sets conduced them to the work of Grothendieck on the generalization of the notion of sheaf  in the context of Category Theory. Lawvere discovered that  sheaves sited on a Grothendieck topology (a Grothendieck topos) admit the basic operations of set theory. In the same period  Tierney realized that Grothendieck's work could be used to approach  an axiomatic study of sheaves. After this (around 1964), working together, Lawvere and Tierney  axiomatise the  Categories of Sheaves of Sets, and then the category of sets,  using a categorial formulation of set-theoretic properties. Later, in 1972, Tierney found that in this context the notion of sheaf allowed to explain Cohen's Forcing and the Solovay-Scott approach in terms of topoi, obtaining a topos which contain a sheaf which lies between the sheaf of the natural numbers and its power sheaf, which despite of  being an intuitionist model can be collapsed  to a classical one. Also during this period, in 1971, Schoenfield  noted that the construction of Scott and Solovay could be adapted to partial orders independently of the Boolean scheme, here is when the form of forcing as we  understand it  nowadays appeared.\\
\indent It  is at this point in history  that forcing created a link between Set theory, Topos theory, and, by the idea  of completing theories by forcing developed by Barwise and Robinson, also with Model Theory. This can be also understood as a bridge between  classical Logic of  classical  set theory,  the intuitionist logic intrinsic to topoi and Kripke's Models, and infinitary  logics which had been studied in Model theory. However the startling aspect implicit in these connections  did not play a fundamental role on further developments, each theory followed mainly  their own motivations, and the different formulations of Cohen's forcing and their intrinsic difficulties did not suffer great variations or improvements, even when numerous and fundamental results were obtained using this technique.\\

In 1995, X. Caicedo  in its outstanding paper, \textit{The Logic of Sheaves of Structures} \cite{caicedo}, introduced a new approximation to  Kripke-Joyal semantics, avoiding the technicalities of Topos and Category theory. In this context he introduced a notion of a generic filter on the variation domain of a sheaf of structures and via these filters he proved the corresponding Generic Model Theorem (a generalization of the generic model  theorem of Forcing and the Łoś theorem on ultraproducts), which allows the collapse of an structure of variable objects to an static limit, unifying  set theoretic forcing constructions and the model existence theorems of classical and infinitary logics like  completeness, compactness, omitting types etc.  The simplicity and elegance of the proofs of these classical results of model theory in this context is remarkably, here the bridge that forcing establishes between Topos Theory, Set theory and Model Theory reappear showing its fundamental character.  In this paper I will show how, using these tools, it is possible to give an alternative proof of the Independence of the Continuum Hypothesis which simplifies and unifies the classical boolean and intuitionistic approaches and at the same time avoids most of the categorical tools of the topoi based approach. It will follow that the approaches of Scott-Solovay and Fitting are just particular cases of the construction of a Cumulative Hierarchy of variable sets over an arbitrary topological space; and instead of considering sheaves within a particular topos we will just need a sheaf which will correspond to the  construction of a Cumulative Hierarchy of variable sets within a topos. 

\section{Sheaves of Structures} 

In this section I will  give a brief summary of some  of the results contained in the paper of X. Caicedo \cite{caicedo};  the main objective will be to introduce the sheaves of structures, to understand the logic that will govern them and to introduce generic models and the fundamental theorem of Model Theory. Most of the results will be given without a proof, a  detailed and elegant introduction about these results is contained in \cite{caicedo}, nevertheless I will try to motivate the definitions and the results.\\

\subsection{The Logic of Sheaves of Structures}\label{lss}

The sheaves of structures will be formed by variable objects, which variate over a structure (a frame) of \textit{"States of Knowledge" }. To motivate the definition we can make the next analogy, consider the frame of variation  as spatio-temporal states, and us as variable objects which variate over these states. Our "truths"  and "properties"  will be conditioned by extensions within these spatio temporal states. This latter aspect is fundamental, opposite to classical logic where the truth is absolute  and independent of the context,  the logic of the variable objects will be determined by a very natural principle:
\begin{center}
 \textit{ "A proposition about an object will be true if it is true in an extension where this object lies."}
\end{center}
We can translate this ideas in mathematical terms very easily, we can think, for example, the extensions as neighbourhoods in a topological space and the objects as functions such that its image  determine  its characteristics in some instant or context. It is here where the notion of sheaf in its more elementary definition, a sheaf over a topological space, become very useful.

\begin{defin}
Let $X$ be a topological space, a sheaf over $X$ is a couple $(E,p)$ where $E$ is a topological space and $P:E\rightarrow X$ is a local homeomorphism, i.e., a function such that for each $e\in E$ there exists a neighbourhood $V$ of $e$ such that:
\begin{enumerate}
\item $p(V)$ is open in $X$
\item $p\upharpoonright_{V}:V\rightarrow P(V)$ is an homeomorphism. 
\end{enumerate}
Given an open set $U$ in $X$, a function $\sigma:U\rightarrow E$ such that  $p\circ \sigma=id_{U}$ is called a local section, if $U=X$, $\sigma$ is called  a global section. The set $p^{-1}(x)\subset E$ for $x\in X$ is called the fibre of $X$.
\end{defin}   

Now using the notion of sheaf we can translate our intuitive ideas in mathematical terms in a very interesting way: Consider the spatio temporal states as a sheaf $(E, p)$ over a space $X$, where $X$ describes for example the causal relations i.e. a structure which describe how events relate with others events, the variable objects will be sections  $\sigma$ on this sheaf in such a way that at each instant $x\in X$ (each node on the causal structure $X$)  $\sigma(x)$ determine what $\sigma$ perceives at the instant $x$, thus $\sigma(x)$ can be understood as the causal past of $x$ and the fiber $p^{-1}(x)$ can be understand as the sets of possible causal pasts of $x$ \footnote{I should say here that this is more than an analogy, the recent work of C. Isham \cite{isham2}, F. Markopoulou \cite{markopoulou} and J. L. Bell \cite{bell} deal with similar ideas in the context of topos theory and Frame-Valued sets. I think that the approach that I am presenting here will be useful also in this context. I have  used this tools on Quantum Mechanics showing how these tools can give us a new interesting perspective of the theory (see \cite{benavides3}) that can be fundamental to a future development of a Quantum Gravity theory.}. Furthermore, we know that if two sections coincide in one point, they coincide in a neighbourhood containing the point (this is a classical result that can be proved without too much difficulty), thus the equality relation satisfies our truth paradigm, two objects are seeing to be equal if and only if they coincide in an extension. With this example in mind we can construct a natural Model Theory of variable objects, the sheaves of structures.

\begin{defin} Given a fix type of structures $\tau=(R^{n},...,f^{m},...,c,...)$ a sheaf of $\tau$-structures
$\frak{A}$ over a topological space $X$ is given by:\\
a-) A sheaf $( E,p)$ over $X$.\\
b-) For each  $x\in X$,   an $L$-structure $\mathfrak{A}_x=(E_x,R_x,...,f_x,...,c_x,...)$.
where  $E_x=p^{-1}(x)$ (the fiber, which could be empty) is  the universe of the $L$-structure  $\mathfrak{A}$, and should satisfy: \\
i. $R^{\mathfrak{A}}=\bigcup_x R_x$ is open in $\bigcup_x
E^{n}_{x}$ seeing as subspace of $E^n$, where $R$ is an n-ary relation symbol.\\
ii.  $f^{\mathfrak{A}}=\bigcup_{x}f_x:\bigcup_x
E_{x}^{m}\rightarrow \bigcup_x E_x$ is a continuous function where $f$ is an $m$-parameter function symbol.\\
iii. $h:X\rightarrow E$  such that  $h(x)=c_x$, where $c$ is a constant symbol, is continuous.
\end{defin}

Given our logical language we need now an interpretation, a semantics of our structures. We can follow the reasoning within our previous example, as we said our objects will be the sections of the sheaves of structures that can represent, for example, observers in spacetime (timelike geodesics) which properties ( or what is ``true" for them) variate over the nodes (a local spacelike slice for example) of the causal structure. Therefore the logical propositions about this objects will be of the form: \textit{The observers $\sigma_1, \sigma_2,...,\sigma_n$ have (or perceive to have ) the property $P$ in the node  $x$}. And the observers will perceive something as "true" at the node $x$ if holds in an extension ( a close past,  the present and  the future) containing the node $x$. Or more precisely:

\begin{defin} 
Let  $L_{\tau}$ be a first order language of type $\tau$. Given a sheaf of structures  $\frak{A}$ of type $\tau$ over $X$, for 
$\varphi(v_1,...,v_n)\in L_{\tau}$ we can define by induction $\frak{A}\Vdash_x
\varphi(\sigma_1,...,\sigma_n)$ ($\frak{A}$ forces 
$\varphi(\sigma_1,...,\sigma_n)$ in $x\in X$ for  the sections
$\sigma_1,...,\sigma_n$ of $\frak{A}$ defined in $x$):\\
1. If $\varphi$ is an atomic formula and $t_1,...,t_k$ are $\tau$-terms: \[\frak{A}\Vdash_x
(t_1=t_2)[\sigma_1,...,\sigma_n]\Leftrightarrow
t_1^{\frak{A}_x}[\sigma_1(x),...,\sigma_n(x)]=t_2^{\frak{A}_x}[\sigma_1(x),...,\sigma_n(x)]\]
\[\frak{A}\Vdash_x
R(t_1,...,t_n)[\sigma_1,...,\sigma_n]\Leftrightarrow
(t_1^{\frak{A}_x}[\sigma_1(x),...,\sigma_n(x)],...,t_n^{\frak{A}_x}[\sigma_1(x),...,\sigma_n(x)])\in
R_x\]

2. $\frak{A}\Vdash_x
(\varphi\wedge\psi)[\sigma_1,...,\sigma_n]\Leftrightarrow
\frak{A}\Vdash_x\varphi[\sigma_1,...,\sigma_n]$ and
$\frak{A}\Vdash_x \psi[\sigma_1,...,\sigma_n]$\\

3. $\frak{A}\Vdash_x
(\varphi\vee\psi)[\sigma_1,...,\sigma_n]\Leftrightarrow
\frak{A}\Vdash_x\varphi[\sigma_1,...,\sigma_n]$ or
$\frak{A}\Vdash_x\psi[\sigma_1,...,\sigma_n]$.\\

4.
$\frak{A}\Vdash_x\neg\varphi[\sigma_1,...,\sigma_n]\Leftrightarrow$
exists $U$ neighbourhood of $x$ such that for all $y\in U$,
$\frak{A}\nVdash_y \varphi[\sigma_1,...,\sigma_n]$.\\

5. $\frak{A}\Vdash_x
(\varphi\rightarrow\psi)[\sigma_1,...,\sigma_n]\Leftrightarrow$
Exists  $U$ neighbourhood of  $x$ such that  for all $y\in U$ if 
$\frak{A}\Vdash_y\varphi[\sigma_1,...,\sigma_n]$ then
$\frak{A}\Vdash_y\psi[\sigma_1,...,\sigma_n]$.\\

6.  $\frak{A}\Vdash_x \exists
v\varphi(v,\sigma_1,...,\sigma_n)\Leftrightarrow$ exists $\sigma$
defined in $x$ such that
$\frak{A}\Vdash_x\varphi(\sigma,\sigma_1,...,\sigma_n)$\\

7. $\frak{A}\Vdash_x\forall
v\varphi(v,\sigma_1,...,\sigma_n)\Leftrightarrow$ exists $U$
neighbourhood of $x$ such that for all $y\in U$ and all $\sigma$
define in $y$,
$\frak{A}\Vdash_y\varphi[\sigma,\sigma_1,...,\sigma_n]$.

(In 4,5,7 \quad $U$ should satisfy $U\subseteq \bigcap_{i}dom(\sigma_i)$)

\end{defin}

The first consequence of the definition is that our principle of "truth"  is maintained, i.e if some proposition is true at a node $x$ it is also true in an extension containing $x$.

\begin{corol} \label{localtruth}
$\frak{A}\Vdash_x \varphi[\sigma_1,...,\sigma_n]$ if and only if there exists a neighbourhood of $x$, $U$, such that $\frak{A}\Vdash_y
\varphi[\sigma_1,...,\sigma_n]$ for all $y\in U$
\end{corol}

Thus, it is clear  that the extensional property  of truth is a fundamental characteristic of these structures, furthermore we can characterize the semantics in terms of extensions instead of nodes obtaining a simple version of the Kripke-Joyal semantics.

\begin{defin} 
Given an open set  $U\subset X$ and $\sigma_i$ sections  define
in $U$, we say that
\[\frak{A}\Vdash_U\varphi[\sigma_1,...,\sigma_n]\Leftrightarrow\text{
for all $x\in U$
}\frak{A}\Vdash_x\varphi[\sigma_1,...,\sigma_n]\]
\end{defin}

\begin{teor}[Kripke-Joyal semantics]

$\frak{A}\Vdash_U\varphi[\sigma_1,...,\sigma_n]$ is defined by:\\
1. If $\varphi$ is an atomic formula:\\
$\frak{A}\Vdash_U\sigma_1=\sigma_2\Leftrightarrow
\sigma_1\upharpoonright_U=\sigma_2\upharpoonright_U$.\\
$\frak{A}\Vdash_U R[\sigma_1,...,\sigma_n]\Leftrightarrow
(\sigma_1,...,\sigma_n)(U)\subseteq R^{\frak{A}}$.\\
2.
$\frak{A}\Vdash_U(\varphi\wedge\psi)[\sigma_1,...,\sigma_n]\Leftrightarrow
\frak{A}\Vdash_U\varphi[\sigma_1,...,\sigma_n]$ and
$\frak{A}\Vdash_U\psi[\sigma_1,...,\sigma_n]$.\\
3.
$\frak{A}\Vdash_U(\varphi\vee\psi)[\sigma_1,...,\sigma_n]\Leftrightarrow$
there exist open sets $V,W$ such that  $U=V\cup W$,
$\frak{A}\Vdash_V\varphi[\sigma_1,...,\sigma_n]$ and
$\frak{A}\Vdash_W\psi[\sigma_1,...,\sigma_n]$.\\
4.
$\frak{A}\Vdash_U\neg\varphi[\sigma_1,...,\sigma_n]\Leftrightarrow$
For all open set  $W\subseteq
U$,\quad$W\neq\emptyset$,\quad$\frak{A}\nVdash_W\varphi[\sigma_1,...,\sigma_n]$.\\
5. $\frak{A}\Vdash_U
\varphi\rightarrow\psi[\sigma_1,...,\sigma_n]\Leftrightarrow$ for all
open set $W\subset U$,\quad if
$\frak{A}\Vdash_W\varphi[\sigma_1,...,\sigma_n]$ then
$\frak{A}\Vdash_W\psi[\sigma_1,...,\sigma_n]$.\\
6. $\frak{A}\Vdash_U\exists
v\varphi(v,\sigma_1,...,\sigma_n)\Leftrightarrow$ there exists  $\{U_i\}_i$  an open cover of  $U$ and $\mu_i$
sections defined on  $U_i$ such that 
$\frak{A}\Vdash_{U_i}\varphi[\mu_i,\sigma_1,...,\sigma_n]$ 
for all $i$.\\
7.
$\frak{A}\Vdash_U\forall\varphi(v,\sigma_1,...,\sigma_n)\Leftrightarrow$
for all open set $W\subset U$ and $\mu$ defined on  $W$,
$\frak{A}\Vdash_W\varphi(\mu,\sigma_1,...,\sigma_n)$.
\end{teor}

The geometrical character of this theorem and the previous definitions show an explicit connection between geometry and logic. For example the result about the existential quantifier in the theorem above shows that the validity of a proposition involving an existential in a open set   $U$, does not  imply the existence of a global open section in this set, but just the existence of a family of sections which satisfy the property locally; this not implies the existence of a global section, but implies the existence of a section which satisfy the property defined in a dense open set contained in  $U$ (see theorem 3.3 in \cite{caicedo}). In the same way the logic also can determine the topological features of the fiber space and viceversa; it can be proved, for example,  that the excluded middle is forced in a node $x$ if and only if there exists a neighbourhood $U$ of $x$ such that in $p^{-1}(U)$ the Hausdorff property (T2) holds. Thus the topology of the ground space and of the fiber space conditioned the logic; a particular trivial case is when the node $x$ is an isolate point in $X$ or when $X=\{x\}$, in theses cases we obtain we obtain classical logic as an special case of the logic of sheaves of structures, or more precisely in these situations for every formula $\varphi$:

\[\frak{A}\Vdash \varphi \leftrightarrow \frak{A}_x\models \varphi.\]
 
The next natural question is what is the logic that rules the sheaves of structures, is any logic we know? The answer is very interesting, we will see that the logic of sheaves of structures is an intermediate logic between  intuitionistic logic and classical logic. This means that,  a proposition is a law of the intuinionistic logic if and only if is forced en each node of every sheaf of structures, but as we see, depending on the topology of the ground space, we know that this logic can get closer to classical logic.\\
To see that intuitionistic logic holds on the sheaves of structures lets remember how Kripke's Models and the respective forcing are defined:

\begin{defin}
$\mathbb{K}=\langle\mathbb{P},\leq,(K_p)_{p\in\mathbb{P}},(f_{pq})_{(p,q)\in\mathbb{P}}\rangle$
is a $\tau$-Kripke model if:\\
i. $(\mathbb{P},\leq)$ is a partial order.\\
ii. For each $p\in\mathbb{P}$, $K_p$ is a classic $\tau$-structure.\\
iii. Given $p,q\in\mathbb{P}$ if $p\leq q$ there exists an homomorphism 
\[f_{pq}:K_p\rightarrow K_q\] such that
$f_{pp}=id_{K_p}$ and if  $t\geq q\geq p$, $f_{qt}\circ
f_{pq}=f_{pt}$
\end{defin}

\begin{defin}[Kripke's Forcing]
Given a Kripke model 
$\mathbb{K}=\langle\mathbb{P},\leq,(K_i)_{i\in\mathbb{P}},(f_{pq})_{(p,q)\in\mathbb{P}}\rangle$
we define for $p\in\mathbb{P}$, 
$\Vdash_p\varphi$ ($\varphi$ is forced at the node $p$) as: \\
1. $\Vdash_p a_1=a_2\Leftrightarrow a_1=a_2$\\
2. $\Vdash_p R(a_1,...,a_n)\Leftrightarrow (a_1,...,a_n)\in
R^{K_p}$\\
3. $\Vdash_p
(\varphi\rightarrow\psi)[a_1,...,a_n]\Leftrightarrow\forall q\geq
p, \Vdash_q\varphi[f_{pq}(a_1),...,f_{pq}(a_n)]\Rightarrow
\Vdash_q\psi[f_{pq}(a_1),...,f_{pq}(a_n)]$.\\
4. $\Vdash_p\neg\varphi[a_1,...,a_n]\Leftrightarrow\forall q\geq
p:\nVdash_q\varphi[f_{pq}(a_1),...,f_{pq}(a_n)]$.\\
5. $\Vdash_p\exists x\varphi(x)[a_1,...,a_n]\Leftrightarrow$
exists $a\in K_p$ such that $\Vdash_p\varphi[a,a_1,...,a_n]$.\\
6. $\Vdash_p\forall x\varphi(x)[a_1,...,a_n]\Leftrightarrow\forall
q\geq p\quad\forall b\in K_q:
\Vdash_q\varphi[b,f_{pq}(a_1),...,f_{pq}(a_n)]$.
\end{defin}

A Kripke model can be seen as a Sheaf of Structures in the next way: In $\mathbb{P}$ consider the topology which has as a base the sets of the form $[r)=\{t\in\mathbb{P}: t\geq r\}$ for $r\in \mathbb{P}$. We write the topology generated by these sets as $\mathbb{P}^{+}$, the open sets of this topology are the hereditary subsets of the order i.e. $\mathbb{P}^{+}=\{S\subset \mathbb{P}:$ for all $i\in S$ if $j\geq i$ then $j\in S\}$.  For each $q\in \mathbb{P}$ take as its fiber the set  $K_q=p^{-1}(q)$. On the set $E=\dot{\bigcup}_{q\in\mathbb{P}}K_q$ we consider the topology generated by the images of the sections $\sigma:=(a_i)_{i\in  U}$ where $a_i\in K_i$ , $U$ is an open set and $f_{ij}(a_i)=a_j$ for $i\leq j$. In this way we get a sheaf of structures $\mathbb{K}^{*}$ of type $\tau$ over the space $X=(\mathbb{P},\mathbb{P}^{+})$. It can be proved that all the inuitionistic laws and all the theorems of the Heyting calculus  hold on every node of all sheafs, and if a proposition holds in all the sheafs particularly holds in a Kripke Model, thus by the completeness theorem (see \cite{van} theorem 5.3.10) it can be deduced by the Heyting calculus.\\

An alternative way to see that intuitionistic logic hold in all the sheafs of structures is  to consider a topological valuation on formulas. Let $\sigma_1,..., \sigma_n$ sections of a sheaf $\mathfrak{A}$ defined on an open set $U$, we define the \textit{"truth value"} of a proposition $\varphi$ in $U$ as:
\[ [[\varphi[\sigma_1,...,\sigma_n] ]]_{U}:=\{x\in U: \mathfrak{A}\Vdash _x \varphi [\sigma_1,...,\sigma_n]\}\]
From corollary \ref{localtruth} we know that $[[\varphi[\sigma_1,...,\sigma_n]]]_{U}$ is an open set, thus we can define a valuation as a topological valuation on formulas:
\[T_{U}:\varphi\mapsto [[\varphi[\sigma_1,...,\sigma_n]]]_{U}\]
Note from the definition of the logic we can define the value of the logic operator in terms of the operations of the algebra of open sets. For instance the proposition $\neg\varphi$ is valid in a point if there exists a neighbourhood of that point where $\varphi$ does not hold at each point, then $\neg\varphi$ holds in the interior of the complement of the set where $\varphi$ holds i.e $[[\neg \varphi]]_{U}= Int ((U\setminus [[\varphi]]_{U})$. Reasoning in analogous way we have:

\begin{enumerate}
\item $[[\sigma_1=\sigma_2]]_{U}=\{x\in U:\sigma_1(x)=\sigma_2(x)\}$.
\item $[[R[\sigma_1,...,\sigma_n]]]_{U}=\{x\in U: (\sigma_1(x),...,\sigma_n(x))\in R^{A}\}$
\item $[[\neg \varphi]]_{U}= Int ((U\setminus [[\varphi]]_{U}),$
\item $[[\varphi\wedge \psi]]_{U}=[[\varphi]]_{U}\cap[[\psi]]_{U},$
\item $[[\varphi\vee \psi]]_{U}=[[\varphi]]_{U}\cup [[\psi]]_{U},$
\item $[[\varphi\rightarrow \psi]]_{U}= Int((U\setminus[[\varphi]]_{U})\cup [[\psi]]_{U}),$
\item $[[\exists u \varphi (u)]]_{U}=\bigcup^{\sigma\in W}{W\subset U}[[\varphi[\sigma]]]_{W},$
\item $[[\forall u\varphi(u)]]_{U}=Int(\bigcap^{\sigma\in W}_{W\subset U}[[\varphi[\sigma]]]_{W}).$
\end{enumerate}

It follows then 

\[\mathfrak{A}\Vdash_{U} \varphi[\sigma_1,...,\sigma_n]\leftrightarrow [[\varphi[\sigma_1,...,\sigma_n]]]_{U}=U.\]
 
Therefore, since the open sets form a complete Heyting algebra, and we have defined the logic operators in terms of the algebra operations, we have that the formulas that get the value $1$ in complete Heyting algebras are forced in every sheaf of structures.    
From these results which relate the sheaves of structures with Kripke's models and intuitionistic logic is becoming clear that the independence results based on Kripke's Model are just particular cases of this more general approach. On the other hand the approach based on Boolean Valued Models of set theory, developed originally by Scott, can be reduced to a sheaf of structures where the \textit{``truth"} values $[[\varphi(\sigma_1,...,\sigma_n]]_{U}$ are clopen sets. However to see how the approach treated here include completely the intuitionistic and Boolean Valued Model approaches we need one more result.
 
\subsection{Generic Models}

As was sketched in the introduction, the notion of generic model and the respective generic model theorem, in the context of sheaves of structures, are the fundamental results that will unify the notion of generic model in set theoretic forcing and  the generic models of Model Theory  introduced by Robinson. In a few words, these two approaches can be seen as the construction of classical structures from  Kripke's Models using a generic filter to connect the semantics of the classical structures with the semantics of the Kripke's model.  From this, as we seen before that Kripke's models are just an special case of our sheaves of structures, the notion of generic model can be easy generalized in this context.  The main advantage will be that the usual conditions of enumerability, that are assumed to show the existence of generic models in the classical approaches, will not be necessary in this more general context.\\

\begin{defin} Let $X$ be a topological space and  $Op(X)$ the set of open sets of $X$ . Let  $\mathcal{F}\subset
Op(X)$, $\mathcal{F}$ is called a filter of open sets  of  $X$ if:\\
i. $X\in\mathcal{F}$.\\
ii. If  $U,V\in\mathcal{F}$ then  $U\cap V\in\mathcal{F}$.\\
iii.Given  $V\in\mathcal{F}$ If $V\subset U$ then
$U\in\mathcal{F}$.
\end{defin}

\begin{defin} 
Let $\frak{A}$ be a sheaf of structures of type $\tau$
over $X$. Let  $\mathcal{F}$ be a  filter of open sets of $X$, we say that $\mathcal{F}$ is a Generic Filter of 
$\frak{A}$ if:\\
i. Given $\varphi(v_1,...,v_n)$ and  $\sigma_1,...,\sigma_n$ arbitrary sections of $\frak{A}$ define on $U\in\mathcal{F}$, there exists 
$W\in\mathcal{F}$ such that 
$\frak{A}\Vdash_W\varphi[\sigma_1,...,\sigma_n]$ or
$\frak{A}\Vdash_W\neg\varphi[\sigma_1,...,\sigma_n]$.\\
ii. Given $\sigma_1,...,\sigma_n$ arbitrary sections of
$\frak{A}$ define  on  $U\in\mathcal{F}$, and 
$\varphi(v,v_1,...,v_n)$ a first order formula. If
$\frak{A}\Vdash_U\exists v\varphi(v,\sigma_1,...,\sigma_n)$
then it exists $W\in\mathcal{F}$ and $\sigma$ defined on $W$ such that
$\frak{A}\Vdash_W\varphi(\sigma,\sigma_1,...,\sigma_n)$.
\end{defin}

Consider now $\frak{A}(U)=\{\sigma: dom(\sigma)=U\}$ the sections with domain $U$. $\frak{A}(U)$ is a structure of the same type of $\frak{A}$ seen as a substructure of $\prod_{x\in U}\frak{A}_x$. If $V\subset U$ there is a natural homomorphism $\rho_{UV}:\frak{A}(U)\rightarrow \frak{A}(V)$,  such that $\rho_{UV}(\sigma)=\sigma\upharpoonright V$. Now we can a construct a new classical $\tau$-structure as the direct limit of the $\frak{A}(U)$ as $U$ variates over a filter $\mathcal{F}$:

\begin{defin}  Let $\frak{A}$ be a sheaf of structures and 
$\mathcal{F}$ a filter over $X$, we can associate a classical model to $\frak{A}$ in the next way:\\
let
\[\frak{A}[\mathcal{F}]=\lim_{\rightarrow
U\in\mathcal{F}}\frak{A}(U)\] i.e
\[\frak{A}[\mathcal{F}]=\dot{\bigcup}_{U\in\mathcal{F}}\frak{A}(U)/_{\sim_{\mathcal{F}}}\]
 where for  $\sigma\in\frak{A}(U)$ and $\mu\in\frak{A}(V)$
 \[\sigma\sim_{\mathcal{F}}\mu\Leftrightarrow\text{ existe
 }W\in\mathcal{F}\text{ tal que
 }\sigma\upharpoonright_W=\mu\upharpoonright_W\]
 Let $[\sigma]$ the class of  $\sigma$. We define the relations and functions in the next way:
 \[([\sigma_1],...,[\sigma_n])\in R^{\frak{A}[\mathcal{F}]}
 \Leftrightarrow\text{ existe
 }U\in\mathcal{F}:(\sigma_1,...,\sigma_n)\in R^{\frak{A}(U)}\]
 \[f^{\frak{A}[\mathcal{F}]}([\sigma_1],...,[\sigma_n])=
 [f^{\frak{A}(U)}(\sigma_1,...,\sigma_n)]\] $\mathcal{F}.$ If $\mathcal{F}$ is a generic filter over $X$ for  $\frak{A}$ we say that 
 $\frak{A}[\mathcal{F}]$ is a generic model. 
\end{defin}

We have then the next fundamental result:

\begin{teor}[The Fundamental Theorem of Model Theory] \label{ftmt}
Let $\mathcal{F}$ be a generic filter over $X$ for 
$\frak{A}$, we have then:
\begin{eqnarray*}
\frak{A}[\mathcal{F}]\models\varphi([\sigma_1],...,[\sigma_n])&\Leftrightarrow&\text{
existe }U\in\mathcal{F}\text{ tal que
}\frak{A}\Vdash_U\varphi^G(\sigma_1,...,\sigma_n)\\
&\Leftrightarrow&\{x\in
X:\frak{A}\Vdash_x\varphi^G(\sigma_1,...,\sigma_n)\}\in\mathcal{F}
\end{eqnarray*}
Where $\varphi^G$ is the Gödel translation\footnote{See \cite{van} chapter 5 or \cite{benavides} section 0.2.} of the formula $\varphi$. 
$\varphi$.
\end{teor}

The name of the theorem is due to the fact that the fundamental theorems of Model Theory can be reinterpreted as the construction of generic models over appropriate sheafs (see \cite{caicedo} section 6). Now the approaches of Fitting and Scott can be interpreted as the construction of generic models over some Kripke's Models and Boolean Valued Models respectively,  that are just especial cases of the construction of the cumulative hierarchy of variable sets over a topological space.

\section{The Cumulative Hierarchy of Variable Sets}  

We will construct now the Cumulative Hierarchy of variable sets, we limit ourselves to the construction of the case based over a partial order with the topology of the hereditary subsets; mainly because to prove the independence of the $CH$ this will be enough. Nevertheless this construction can be done over arbitrary topological spaces see \cite{caicedo} or \cite{benavides} section 1.7. We will obtain a sheaf of structures over a partial order $\mathbb{P}$, such that any fiber $V(p)$ will be an intuitionistic model of $ZFC$\footnote{We will see that some  Axioms of $ZFC$ are not forced in its standard form but in alternatively intuitionistic versions that result to be classical equivalent to the standard ones. } that can be collapsed to classical models via the fundamental theorem of Model Theory. The results presented in this section  were originally developed in \cite{villa}. The part concerning the axiom of choice was developed in \cite{rive} and I introduce some results about functions that will be useful in the next section.

\subsection{Definition and Elementary Properties}

To introduce the definition of the cumulative hierarchy we can use the interpretation that categorist use to introduce  the objects of $SET^{\mathbb{P}}$, the advantage here is that the motivation translate literally in the definition; we obtain a truly structure of variable  sets constructed over our classical $\in$ defining  a truly belonging relation between variable sets and not arrows between objects.  Hilbert can rest in peace, even if we will leave Cantor's Paradise we will arrive to Cantor's Heaven avoiding the creepy commutative-diagram purgatory. \\

Using the comprehension axiom of set theory, given a proposition $\varphi(x)$, for any set $A$ we can construct a set $B$ such that $x\in B$ if and only if $x\in A$ and $\varphi(x)$ is "truth" for $x$. As we have seen the notion of truth in our sheaves of structures is not absolute but is based in the extensional truth paradigm. To see how this notion of truth translate when we talk about sets lets consider the next example: Let $\varphi(x)$ be the next proposition \textit{`` x is an even number greater or equal than 4, and, x can be written as the sum of two prime numbers"} . In this moment at my ubication on spacetime $p:=$\textit{``between the 21:00 and the 22:00, of the 5 March 2011, in some place at London"} I cannot assure that $\{z\in\mathbb{N}:\varphi(x)\}=\{x\in \mathbb{N}: x\geq 4 \wedge (x$ is even $)\}$, i.e. that the Goldbach Conjecture is true. However $\varphi$ can be used to get a set at the node $p$,

\[\varphi(p)=\{x:\varphi(x) \text{ holds at the node } p.\}\] 

The Goldbach conjecture has been verified for a huge number of even numbers, and this set of numbers will keep growing each time we will find a new prime number. Probably (maybe not) some day the Goldbach Conjecture will be proved or disproved, but the important feature that follow from this example is that instead of conceiving sets as absolute entities, we can conceive them as variable structures which variate over our \textit{Library of the states of Knowledge}. It is natural then to conceive the set of nodes over our states of Knowledge variates, as nodes in a partial order or points in a topological space that can represent, for example,  the causal structure of spacetime. Our "states Knowledge" will be then structures that represent the sets as we see them in our nodes. Therefore, from each node we will see arise a cumulative Hierarchy of variable sets, which structure will be conditioned by the perception of the variable structures at the other nodes which relate with it. Or more precisely:

\begin{defin} 
Let $( \mathbb{P};\leq )$ be a partial order, for  $p\in
\mathbb{P}$ we define recursively :
\begin{center}{$V_0(p)=\emptyset$}
\end{center}

\begin{center}{
$V_{\lambda}(p)=\bigcup_{\alpha \leq \lambda}V_{\alpha}(p)$ if
$\lambda$ is a limit ordinal.\\}
\end{center}

\begin{center}{$V_{\alpha + 1}(p)=\{f:[p)\rightarrow \bigcup_{q\geq
p}\emph{P}(V_{\alpha}(q)) :$ if $p\leq q \rightarrow  f(q) \subset
V_{\alpha}(q)$, and,  if $p\leq q\leq r \rightarrow \forall g \in f(q)
(g\upharpoonright _{[r)} \in f(r))\}$\\ }
\end{center}

We define then ``the universe at the event p",  as
$V(p)=\bigcup_{\alpha \in On}V_{\alpha}(p)$ and we denote by $V^{\mathbb{P}}$ the sheaf of structures of the variable sets over $\mathbb{P}$ with the topology of the hereditary  subsets.
\end{defin} 
 
For each $p\in \mathbb{P}$ and $\alpha\in On$, $V_{\alpha}(p)$ is a set of functions defined on $[p)$ which values for $q\in [p)$ are sets of functions defined on $[q)$ which values for $r\in [q)$ are sets of functions on $[r)$ and so on. The next step will be to define the belonging relation. Thus in the more natural way we define:

\[\Vdash_pa\in f \text{ (in the node } p, a \text{  belongs to }f \text{ is forced)}\Leftrightarrow a\in f(p),\]

where the $\in$ in the right side is just the classical belonging relation between sets. Now we have defined the belonging relation, our first result will be to show that if the partial order is not trivial, i.e  if it has at lest two related elements, then the cumulative hierarchy $V(p)$, where $p$ is a point which is related  with a different element, is not isomorphic to the classical Cumulative hierarchy $V$.

\begin{lema}
Let $(\mathbb{P},\leq)$ a partial order with at least two different elements related. Let $p, q\in \mathbb{P}$ such that $p\neq q$ and $p\leq q$, then $V(p)$ with the belonging relation defined above is not isomorphic to the Classical Cumulative hierarchy of set theory $V$.
\end{lema}

\begin{proof} 
 $V_0(p)=\emptyset =V_0(q)$ and
$V_1(p)=\{f:[p)\rightarrow \bigcup_{q\geq p}\{\emptyset \} \}=
\{f:[p)\rightarrow \{ \emptyset \} : f(s)=\emptyset$ for all
$s\} $. As in the classical cumulative hierarchy\footnote{Remember that the classical cumulative hierarchy is constructed in the next way $V_0=\emptyset$, $V_{\alpha +1 }=P(V_\alpha)$ and $V_{\lambda}=\bigcup_{\mu \leq \lambda}V_{\mu}$ if $\lambda$ is a limit ordinal.} $V_1(p)$ has just one element; but in the next step we get
\[V_2(p)=\{V_0\upharpoonright_{[p)},V_1\upharpoonright_{[p)},\{(p,\emptyset),(q,\{f\}),...\}...\}\]
where $f$ is the only element of $V_1(p)$. Thus $V_2(p)$
has at least three elements, but  $V_2$ in the classical hierarchy has just 2.
\end{proof}

On the other hand if a point $p$ is related just with itself or is a maximal point, the classical universe $V$ and $V(p)$ are isomorphic.

\begin{lema}
Let, $(P,\leq)$ a partial order. Let $p\in \mathbb{P}$ such that  $(p,q)\in \leq \Leftrightarrow q=p$ then $V(p)$ is isomorphic to $V$.
\end{lema}

\begin{proof}
Let $\psi:V(p)\rightarrow V$ given by 
\[\psi(V_0)=\emptyset\]\[\psi(g)=\{\psi(h):h\in g(p)\}.\] 
From the definition follows that  $\Vdash_p f \in g$ iff $f\in g(p)$ iff $\psi(f)\in \psi(g)$ then $\psi$ is a morphism. We have to see now that $\psi$ is also a bijection. Let $ran_p(f)=min\{\alpha \in On: f\in V_{\alpha+1}(p)\}$ the rank of $f$ at $p$.. If $\Vdash_p x\in f$ then $x\in f(p)$, and if $\alpha=rank_p(f)$ we have $f\in V_{\alpha +1}(p)$, therefore $x\in f(p)\subset V_{\alpha}(p)$ thus 
\begin{equation}
rank_p(x)< rank_p{f}=\alpha.\label{rank}
\end{equation}
 Using this last result we want to show that $\psi$ is injective by induction on the rank. Our induction hypothesis will be the following one: Given $\alpha$, for $x, y \in V(p)$ if $x\neq y$ and $rank_p(x)<\alpha$, $rank_p(y)<\alpha$, then $\psi(x)\neq \psi(y)$. If $\alpha=0$ there is nothing to prove.  Lets suppose then that the result holds for $\alpha>0$. Let $f,g\in V(p)$ if $\psi(f)=\psi(g)$ then $\{\psi(h):h\in f(p)\}=\{\psi(m): m\in g\}$, we have then that given $h\in f(p)$ there exists $m\in G(p)$ such that $\psi(h)=\psi(m)$, but by \ref{rank} we have that the rank of $m$ and $h$ is less than $\alpha$ thus by the inductive hypothesis $h=m$ and we can conclude that $f=g$.\\
On the other hand given an arbitrary set $B$, let
 \[\widehat{B}:[p)=\{p\}\rightarrow\bigcup_{q\geq
p}V(q)=V(p)\]\[\widehat{B}(p)=\{\widehat{a}:a\in B\}\]
Thus $\widehat{\emptyset}(p)=\emptyset$. By induction on the rank we want to show that $\psi(\widehat{B})=B$. If $rank_p(B)=0$ then $B=\emptyset=\widehat{\emptyset }(p)$. Let $B$ such that $rank_p(B)=\alpha>0$, we have  

\begin{eqnarray*}
\psi(\widehat{B})&=&\{\psi(h):h\in\widehat{B}(p)\}\\
&=&\{\psi(h):h\in\{\widehat{b}:b\in B\}\}\\
&=&\{\psi(\widehat{b}):b\in B\}\\
&=&\{b:b\in B\}=B \text{(by the inductive hypothesis)}
\end{eqnarray*}
\end{proof}

Note that  the rank does not grow with the order,  if $p<q$ then $rank_p(f)\geq  rank_q(f\upharpoonright _{[q)} $ because if $f\in V_{\alpha}(p)$ then $f\upharpoonright _{[q)}\in V_{\alpha}(q)$ by the second condition in the definition of the hierarchy. We have also that $V_{\alpha}\upharpoonright_{[p)}\in V_{\alpha+1}(p)$ and:

\begin{lema} Given  $p\in\mathbb{P}$  we have \[ V_{0}(p)\subset
V_{1}(p)\subset V_{2}(p)\subset...\subset V_{\beta}(p)\subset
V_{\beta+1}(p)\subset...\]
\end{lema}

\begin{proof} By induction  on $\alpha$  we will see that $V_{\alpha}(p)\subset V_{\alpha+1}(p)$:\\
1. If  $\alpha=0$ is evident.\\
2. If $\alpha$ is a limit ordinal and  $f\in
V_{\alpha}(p)=\bigcup_{\beta<\alpha}V_{\beta}(p)$; then
$range(f)\subset\bigcup_{q\geq p}\textsl{P}(V_{\beta}(q))$ for some $\beta<\alpha$. Since 
$V_{\beta}(q)\subset\bigcup_{\gamma<\alpha}V_{\gamma}(q)=V_{\alpha}(q)$.
and  $range(f)\subset\bigcup_{q\geq
p}\textsl{P}(V_{\alpha}(q))$, we have
$f(q)\subset V_{\alpha}(q)$. The second condition of the definition follows from the fact that  $f\in V_{\xi}(p)$ for some $\xi<\alpha$ and by the inductive hypothesis, $f\in V_{\xi+1}(p)$. Thus $f\in V_{\alpha+1}(p)$.\\
3. If  $\alpha=\gamma+1$ y $f\in V_{\alpha}(p)$, then
$range(f)\subset\bigcup_{q\geq p}\textsl{P}(V_{\gamma}(q))$. By the inductive hypothesis $V_{\gamma}(q)\subset
V_{\gamma+1}(q)=V_{\alpha}(q)$, for all $q\geq p$, thus
$\bigcup_{q\geq p}\textsl{P}(V_{\gamma}(q))\subset\bigcup_{q\geq
p}\textsl{P}(V_{\alpha}(q))$, then, $imagen(f)\subset\bigcup_{q\geq
p}\textsl{P}(V_{\alpha}(q))$. Again we can conclude $f\in V_{\alpha+1}(p)$.
\end{proof}

Even if generally $V(p)$ and $V$ are not isomorphic there is an standard way to embed $V$ in $V(p)$. Given an arbitrary set $a$ let 

\[\widehat{a}(p):[p)\rightarrow\bigcup_{q\geq
p}V(q)\]
\[\widehat{a}(p)(q)=\{\widehat{b}(p)\upharpoonright_{[q)}:b\in
a\}\]

\begin{lema} 
Given an arbitrary set $a$, $p,q\in\mathbb{P}$ such that $q\geq p$ we have
\[\widehat{a}(p)\upharpoonright_{[q)}=\widehat{a}(q)\]
\end{lema}

\begin{proof} 
We prove the result by induction on the rank of $a$ in the classical cumulative hierarchy:\\
i. Si $rank(a)=0$,  $a=\emptyset$, then
\[\widehat{\emptyset}(p)(r)=\emptyset\] for all $r\geq q\geq p$, and
\[\widehat{\emptyset}(q)(r)=\emptyset=\widehat{\emptyset}(p)\upharpoonright_{[q)}(r)\]\\
ii. Given $\alpha>0$ suppose the result holds for any set with rank less than $\alpha$. Let $a$ be a set with $rank(a)=\alpha$. Thus, given $r\geq q\geq p$
\begin{eqnarray*}
\widehat{a}(p)\upharpoonright_{[q)}(r)&=&\{\widehat{b}(p)\upharpoonright_{[r)}:b\in
a\}\\
&=&\{\widehat{b}(r):b\in a\}\text{ (by the inductive hypothesis)}\\
&=&\{\widehat{b}(q)\upharpoonright_{[r)}:b\in a\}\text{ (by the inductive hypothesis)}\\
&=&\widehat{a}(q)(r)
\end{eqnarray*}
\end{proof}

\begin{corol}\label{redef}
\[\widehat{a}(p)(q)=\{\widehat{b}(q):b\in a\}\]
\end{corol}

It should be clear now how we should define the embedding of $V$ in $V(p)$.

\begin{teor}
 \[\Psi:V\rightarrow V_(p)\]
\[\Psi(a)=\widehat{a}(p)\]
is a monomorphism. 
\end{teor}

\begin{proof}
Step 1:  $a\neq b$ iff $\widehat{a}(p)\neq \widehat{b}(p)$ ($\Psi$ is injective).\\
We proceed by induction on the rank of $a$ and $b$. If $rank(a)=0=rank(b)$, $a=\emptyset=b$ then the result follows because the premise is false in both directions. Now consider $c$ and $d$ two sets with rank $\alpha>0$.  If 
\[\widehat{c}(p)(q)=\{\widehat{x}(q):x\in c\}=\{\widehat{y}(q):y\in d\}=\widehat{d}(p)(q)\]
for all $q\geq p$, then for $x\in c$, there exists $y\in d$ such that $\widehat {x}(q)=\widehat{y}(q)$, since $x$ and $y$ are elements of $c$ and $d$ respectively, their rank is less than $\alpha$; thus by the inductive hypothesis we can conclude that $x=y$ and then $c=d$.\\
Step 2: $b\in a$ iff $\widehat{b}(p)\in \widehat{a}(p)(p)$ ($\Psi$ is a morphism).\\
If $a\in b$ by the corollary \ref{redef} $\widehat{a}(p)\in \widehat{b}(p)(p)$. On the other hand if $\widehat{a}(p)\in\widehat{b}(p)(p)=\{\widehat{c}(p):c\in b\}$ then $\widehat{a}(p)=\widehat{c}(p)$ for some  $c\in b$,  but by step 1 this implies  $a=c$, thus $a\in b$.\\
Step 3: $\widehat{a}(p)\in V(p)$.\\
Again, we proceed by induction on the classical hierarchy:\\
 If the rank of  $a$ is zero,  $a=\emptyset$ and
\[\widehat{\emptyset}(p):[p)\rightarrow\bigcup_{q\geq p}V(q)\]
\[\widehat{\emptyset}(p)(q)=\emptyset.\]
Therefore $\widehat{\emptyset}(p)\in V_1(p)$.\\
Let $\alpha>0$ and suppose that for every $b$ with rank less than
 $\alpha$, $\widehat{b}(p)\in V(p)$. Let $a$ be a set such that $rank(a)=\alpha$.
\[\widehat{a}(p):[p)\rightarrow\bigcup_{q\geq
p}V(q)\]
\[\widehat{a}(p)(q)=\{\widehat{b}(p)\upharpoonright_{[q)}:b\in
a\}\]  
We see that  $\widehat{a}(p)(q)\subset V(q)$ because given
$b\in a$ as $ran(b)<ran(a)$, by the inductive hypothesis 
$\widehat{b}(p)\in V(p)$ and then
$\widehat{b}(p)\upharpoonright_{[q)}\in V(q)$. On the other hand given
$r\geq q\geq p$ and $g\in \widehat{a}(p)(q)$, we have that 
$g=\widehat{b}(q)$, for some $b\in a$;  therefore 
$\widehat{b}(q)\upharpoonright_{[r)}=\widehat{b}(r)\in
\widehat{a}(p)(r)$. From this $\widehat{a}(p)$ satisfy the conditions in the definition of the cumulative hierarchy and we can conclude that $\widehat{a}(p)\in V(p)$.
\end{proof}

\subsection{ZFC in $V(p)$}
 
Now, we want to show that the axioms of Zermelo-Fraenkel plus the Axiom of choice are forced on $V(p)$ for each $p\in\mathbb{P}$. Most of the axioms will be forced in its classical standard definition, nevertheless the Axiom of Choice and the Axiom of Foundation will be forced in some  intuitionistic weaker versions that are classical equivalent to the standard ones. We will see also  how some fundamental constructions of set theory work on $V(p)$.\\

\subsection*{ \normalsize{Axiom of Existence: \textit{There exists a set without elements.}}} 

\begin{lema}
\[\Vdash_{p} \exists x \forall y(y\notin x)\] for every $p\in\mathbb{P}$
\end{lema}

\begin{proof}
For every $q\geq p$ and $f\in V(q)$ if $r\geq q$, $\nVdash f\in \widehat{\emptyset}(r)$ because $\widehat{\emptyset}(r)(r)=\emptyset$
\end{proof}

\subsection*{\normalsize{Axiom of Extensionality: \textit{Two sets are equal if they have the same elements.}}}

\begin{lema}
 \[\Vdash_p \forall x\forall y (\forall z(z\in x\leftrightarrow z\in y)\rightarrow x=y)\]
\end{lema}  
 
\begin{proof}
let $q\geq p$ and $f,g\in V(q)$, suppose that $z\in f(r)$ ($\Vdash_r z\in f$) if and only if $z\in g(r)$ ($\Vdash_r z\in g$) where $r\geq q\geq p$. Thus,  by extensionality on the external model,  $f(r)=g(r)$. Therefore  for every $r\geq q\geq p$ we have $\Vdash_r f=g$.
\end{proof} 

\subsection*{\normalsize{Axiom of Comprehension: \textit{ Let $\varphi(x,z,w_1,...,w_n)$ be a formula. For every set $z$ there exists a set $y$ such that, $x\in y$ if and only if $x\in z$ and $\varphi(x)$ holds for $x$ }}}

\begin{lema}
\[\Vdash_p \forall z\forall w_1,...,w_n \exists y \forall x(x\in y \leftrightarrow x\in z \wedge \varphi)\]
\end{lema}

\begin{proof}
Let $z, w_1,...,w_n \in V(p)$ and 
\[y:[p)\rightarrow  \bigcup_{q\geq p} V(q)\]
\[y(q)=\{x\in z(q); \Vdash_q \varphi(x).\}\]
Let $\alpha=ran_p (z)$ then $z\in V_{\alpha+1}(p)$. Since $y(q)\subset z(q)$ then $y(q)\subset V_{\alpha}(q)$, therefore 
\[y:[p)\rightarrow \bigcup_{q\geq p} P(V_{\alpha}(q)).\]
Given $r\geq q\geq p$ and $g\in y(q)$ we have $g\in z(q)$ and $\Vdash_q \varphi(g)$. Thus $g\upharpoonright_{[r)}\in z(r)$ and $\Vdash_r \varphi(g\upharpoonright_{[r)})$. Then $g\upharpoonright_{[r)}\in y(r)$ and $y\in V(p)$. 
\end{proof}

\subsection*{\normalsize{ Axiom of Pairing: \textit{Given two sets $x$ and $y$, there exists a set $z$ such that $w\in z$ if and only if $w=x$ or $w=y$.}}}

\begin{lema}
\[\Vdash_p \forall x\forall y\exists z(\forall w(w\in z \leftrightarrow w=x\vee w=y)\]
\end{lema}

\begin{proof}
Given $q\geq p$ and $x,y\in V(q)$, let 
\[z:[q)\rightarrow \bigcup_{q\geq p} V(q)\]
\[z(r)=\{x\upharpoonright_{[r)}, y\upharpoonright_{[r)}\}.\]
it is clear that $\Vdash w\in z$ if and only if $w=x\upharpoonright_{[r)}$ or $w=y\upharpoonright_{[r)}$. Therefore $\Vdash_r \forall w(w\in z \rightarrow (w=x\vee w=y))$ the other direction is clear. To conclude is enough to show that $z\in V(p)$. We have $z(q)\subset V(q)$ and given $r\geq q\geq p$ if $w\in z(q)$ then $w=x\upharpoonright_{[q)}$ or $w=y\upharpoonright_{[q)}$. Take for example that $w=x\upharpoonright_{[q)}$ then $w\upharpoonright_{[r)}=(x\upharpoonright_{[q)})\upharpoonright_{[r)}=x\upharpoonright_{[r)}\in z(r)$. Thus $z\in V(p)$. 
\end{proof}

\subsection*{\normalsize{ Axiom of Union: \textit{For any set $\mathcal{F}$, there exists a set $A$ such that $x\in A$ if and only if $x\in Y$ for some $Y\in\mathcal{F}$.}}}

\begin{lema}
$\Vdash_p \forall \mathcal{F}\exists A \forall Y \forall x(x\in Y\wedge Y\in \mathcal{F}\rightarrow x\in A)$ 
\end{lema}

\begin{proof}
Given $q\geq p$ and $\mathcal{F}\in V(q)$ let
\[ A:[q)\rightarrow \bigcup_{q\geq p} V(q)\]
\[A(q)=\bigcup\{Y(q):Y\in \mathcal{F}(q)\}.\]
Thus given $r\geq q$ and $Y, x\in V(r)$ if $\Vdash_{r}x\in Y$ and $\Vdash_r Y\in\mathcal{F}$ then $x\in Y(r)$ and $Y\in\mathcal{F}(r)$. Therefore $x\in \bigcup\{W(r):W\in \mathcal{F}(r)\}=A(r)$,then $\Vdash_r x\in A$.\\
Again we have to show that $A\in V(p)$. Let $\alpha=ran_p(\mathcal{F})$ then $\mathcal{F}\in V_{\alpha+1}(p)$. From the latter given $q\geq p$ we have $\mathcal{F}(q)\subset V_{\alpha}(q)$, thus $A(q)=\bigcup\{Y(q):Y\in \mathcal{F}(q)\}\subset V_{\alpha}(q)$. Let $r\geq q\geq p$ and $h\ A(q)$, then there exists $Y\in \mathcal{F}(q)$ such that $h\in Y(q)$. Since $\mathcal{F}\in V(p)$, $Y\upharpoonright_{[r)}\in \mathcal{F}(r)$. Using that $Y\in V(q)$ we have $h\upharpoonright_{[r)}\in Y(r)$ then $h\upharpoonright_{[r)}\in A(r)$.
\end{proof}

\subsection*{\normalsize{ Axiom of Power Set: \textit{For every set $x$ there exists a set $y$ such that $z\in y$ if and only if $z\subset x$.}}}

\begin{lema}
\[\Vdash_p \forall x\exists y \forall z( z\in y \leftrightarrow z\subset x)\]
\end{lema}

\begin{proof}
Given $f\in V(p)$ the power set $g$ of $f$ is defined as:
\[g:[p)\rightarrow \bigcup_{q\geq p} V(q)\]
\begin{align*}
g(q)=&\{h:[q)\rightarrow \bigcup_{r\geq q}P(f(r)):1. \text{ If } r\geq q, h(r)\subset f(r),\\
&\text{ 2. If }t\geq r\geq q \text{ for all } i\in h(r), i\upharpoonright_{[t)}\in h(t)\}.
\end{align*}
Note that $x\in g(q)$ if and only if given $t\geq r\geq q$,  $y\in V(r)$,  $y\in x(t)$ implies $y\in f(t)$. Therefore if we show that $g\in V(p)$ we will have $\Vdash_p x\in g \leftrightarrow \forall y(y\in x\rightarrow y\in f)$.  To show that $g\in V(p)$, let $\alpha=ran_p(f)$.  If $q\geq p$ since  $f\in V_{\alpha+1}(p)$ we have  $f(r)\subset V_{\alpha+1}(r)$, thus $P(f(r))\subset P(V_{\alpha+1}(r))$ for $r\geq q\geq p$. On the other hand if $h\in g(q)$ it is clear that $h\upharpoonright_{[r)}\in g(r)$. 
\end{proof}

Note that as in the classical hierarchy we have 

\[\Vdash_p P(V_{\alpha})=V_{\alpha+1}.\]

\subsection*{\normalsize{The Axiom of Infinity:\textit{ There exists an Inductive set. }}}

We begin defining the successor function. Let $f\in V(p)$ we define $Suc(f)$ as:
\[Suc(f):[p)\rightarrow \bigcup_{q\geq p} V(q)\]
\[Suc(f)(q)=\{f\upharpoonright _{[q)}\}\cup f(q).\]
Since $f\in V(p)$  we have $f\upharpoonright_{[q)}\in V(q)$ and $f(q)\subset V(q)$ thus $Suc(f)(q)\subset V(q)$. On the other hand if $h\in Suc(f)(q)$ then $h=f\upharpoonright_{[q)}$ or $h\in f(q)$; in the first case $h\upharpoonright_{[r)}=(f\upharpoonright_{[q)})\upharpoonright_{[r)}=f\upharpoonright_{[r)}\in Suc(f)(r)$, in the second case $h\upharpoonright_{[r)}\in f(r)\subset Suc(f)(r)$. Thus $Suc(f)\in V(p)$. Furthermore it is clear that $\Vdash_p x\in Suc(f)\leftrightarrow x=f \vee x\in f$.\\
Recall that the set that $\widehat{\emptyset}(p)$ is forced as the empty set in $p$. We have then that  $f\in V(p)$ is an inductive set if:
\[\Vdash_{p} \widehat{\emptyset}(p)\in f \vee (\forall x(x\in f \rightarrow Suc(x)\in f))\]

\begin{lema}
$\widehat{\omega}(p)$ is an inductive set i.e. 
\[\Vdash _p (\widehat{\emptyset}(p) \in \widehat{\omega}(p)) \wedge
(\forall x(x\in \widehat{\omega}(p) \rightarrow Suc(x) \in
\widehat{\omega}(p)))\]
\end{lema}

\begin{proof}
1. Since $0\in\omega$ we have  $\widehat{0}(p)=\widehat{\emptyset}(p)\in\widehat{\omega}(p)(p)$ then $\Vdash_p \widehat{\emptyset}(p)\in \widehat{\omega}(p)$.\\
2. Let $t\geq q\geq p$ and $x\in V(q)$. Suppose $\Vdash_t x\in\widehat{\omega}(t)$ then $x\in \widehat{\omega}(t)(t)=\{\widehat{n}(t):n\in\omega\}$, thus $x=\widehat{n}(t)$ for some $n\in \omega$. From this we have:
\begin{eqnarray*}
 Suc(\widehat{n}(t))(r)&=&\{ \widehat{n}(r)\} \cup \widehat{n}(t)(r)\\
                    &=&\{\widehat{n}(r)\} \cup
                    \{\widehat{m}(r):m<n\}\\
                    &=&\{\widehat{a}(r):a<n+1\}\\
                    &=&\widehat{n+1}(t)(r).
                    \end{eqnarray*}
                    
Since  $\Vdash_t \widehat{n+1}(t)\in \widehat{\omega}(t)$ we can conclude  $\Vdash_t Suc(x)\in\widehat{\omega}(t)$.                 
\end{proof}

\begin{corol}
\[\Vdash_p \exists x(\widehat{\emptyset}(p)\in x\wedge (\forall y(y\in x\rightarrow Suc(y)\in x))\]
\end{corol}

\subsection*{ \normalsize{Axiom Schema of Replacement: \textit{Let $\varphi$ a formula with free variables $x,y A,w_1,...,w_n$ such that for every $x$ there exists a unique set $y$ such that $\varphi(x,y)$ holds. For every set $A$ there exists a set $Y$ such that for every $x\in A$ there exists $y\in Y$ such that $\varphi(x,y)$ holds.}}}

\begin{lema}
\[\Vdash_p \forall A\forall w_1...\forall w_n(\forall x\in A \exists! y \varphi (x,y,...)\rightarrow \exists Y\forall x\in A\exists y\in Y \varphi(x,y,...))\]
\end{lema}

\begin{proof}
Define 
\[Y:[p)\rightarrow \bigcup_{q\geq p} V(q)\]
\[Y(q)=\{y: \text{ there exists } x\in A(q) \text{ such that } \Vdash_q \varphi(x,y,...).\]
It is easy to verify that $Y$ is the set we are looking for, again we just have to show that $Y\in V(p)$. Let $y\in Y(q)$ then there exists $x\in A(q)$ such that $\Vdash_q \varphi(x,y,...)$ this implies that $y\in V(q)$ then $Y(q)\subset V(q)$. On the other hand given $r\geq q\geq p$, if $g\in Y(q)$ let $x\in A(q)$ such that $\Vdash_q \varphi(x,g,...)$ then $\Vdash_r \varphi(x\upharpoonright_{[r)}, g\upharpoonright_{[r)},...)$. Since $x\upharpoonright_{[r)}\in A(r)$ (because $A\in V(q)$) we have $g\upharpoonright_{[r)}\in \{y: \text{there exists }x\in A(r) \text{ such that } \Vdash_r\varphi(x,y,...)\}=Y(r)$.
\end{proof}

\subsection*{\normalsize{Functions on $V(p)$.}}

Before proving that the axioms of foundation and choice are forced we will need to see how the functions are defined on $V(p)$ or in other words what kind of objects are forced as functions between objects in $V(p)$.\\

Given $f,g\in V(p)$ lets define $(f,g)$ as:
\[(f,g):[p)\rightarrow \bigcup_{q\geq p}V(p)\]
\[(f,q)(q):\{\{f\}\upharpoonright_{[q)},\{f,g\}\upharpoonright_{[q)}\}.\]

We have then:

\begin{lema}
\[\Vdash_p \forall x(x\in (f,g)\leftrightarrow (x=\{f\}\vee x=\{f,g\}))\]
\end{lema}
 
\begin{proof}
Let $q\geq p$ then $\Vdash_q x\in (f,g)$ implies  $x\in (f,g)(q)=\{\{f\}\upharpoonright_{[q)},\{f,g\}\upharpoonright_{[q)}\}$. Thus $\Vdash_q x\in (f,g)$ if and only if $x=\{f\}\upharpoonright_{[q)}$ or $x=\{f,g\}\upharpoonright_{[q)}$ and this is true if and only if $\Vdash_q x=\{f\}\vee x=\{f,g\}$
\end{proof}

From this follows the next important result:

\begin{lema}
Given  two arbitrary sets $a,b$ we have 
\[\Vdash_p\widehat{(a,b)}(p)=(\widehat{a}(p), \widehat{b}(p))\]
\end{lema}

\begin{proof}
\[\widehat{(a,b)}(p):[p)\rightarrow \bigcup_{q\geq p} V(q)\]
\[\widehat{(a,b)}(p)(q)=\{\widehat{\{a\}}(q),
\widehat{\{a,b\}}(q)\}\] 
On the other hand
\[(\widehat{a}(p),\widehat{b}(p))(q)=\{\{\widehat{a}(p)\}\upharpoonright_{[q)},
\{\widehat{a}(p),\widehat{b}(p)\}\upharpoonright_{[q)}\}.\] 
Thus, it is enough to show :
1.$\widehat{\{a\}}(q)=\{\widehat{a}(p)\}\upharpoonright_{[q)}$ and
2.$\widehat{\{a,b\}}(q)=\{\widehat{a}(p),\widehat{b}(p)\}\upharpoonright_{[q)}$.\\
1.$\widehat{\{a\}}(q)(r)=\{\widehat{a}(r)\}=\{\widehat{a}(p)\upharpoonright_{[r)}\}=\{\widehat{a}(p)\}\upharpoonright_{[q)}(r)$.\\
2.$\widehat{\{a.b\}}(q)(r)=\{\widehat{a}(r),\widehat{b}(r)\}=\{\widehat{a}(p)\upharpoonright_{[r)},\widehat{b}(p)\upharpoonright_{[r)}\}
=\{\widehat{a}(p),\widehat{b}(p)\}\upharpoonright_{[q)}(r)$.
\end{proof}

The next step is to define $f\times g$:

\[f\times g:[p)\rightarrow \bigcup_{q\geq p} V(q)\]
\[(f\times g)(q)=\{(a,b)\upharpoonright_{[q)}:a\in f(q) \vee b\in g(q)\}.\]

It is clear that

\begin{lema}\[\Vdash_p \forall x(x\in F\rightarrow \exists a\exists
b(a\in f\wedge b\in g\wedge x=(a,b))).\]\end{lema}

We have also the next results:\\

\begin{lema}
Let $A$ and $B$ be two arbitrary sets then:
\[\Vdash_p
\widehat{A}(p)\times \widehat{B}(p)=\widehat{(A\times
B)}(p)\]
\end{lema}

\begin{proof}
\[(\widehat{A}(p)\times
\widehat{B}(p))(q)=\{(a,b)\upharpoonright_{[q)}:a\in
\widehat{A}(p)(q)\wedge \widehat{B}(p)(q)\}\]
\[=\{(\widehat{m}(q),\widehat{n}(q))\upharpoonright_{[q)}:m\in
A\wedge n\in B\}\]
\[=\{\widehat{(m,n)}(q):m\in A \wedge n\in B\}=\widehat{A\times
B}(p)(q) \]
\end{proof}

\begin{corol} 
 \[\Vdash_p \widehat{\omega}(p)\times\widehat{P(\omega)}(p)=\widehat{\small{(\omega \times
P(\omega))}}(p)\] \[\Vdash_p
\widehat{\textsl{P}(\omega)}(p)\times\widehat{\textsl{P}(\textsl{P}(\omega))}(p)=
\widehat{(\small{\text{P}(\omega)\times\textsl{P}(\textsl{P}(\omega)))}}(p)\]
\end{corol}

Given $A,B\in V(p)$ an object $f$ which is to be forced as a function from $A$ to $B$ has to satisfy:
\[\Vdash_p\forall x(x\in f\rightarrow x\in A\times
B)\] 
\[\Vdash_p ``f \text{ is a function".}\]

Therefore $f$ has to satisfy:

\[\Vdash_p\forall x(x\in f\rightarrow x\in A\times
B)\]

and

\[\Vdash_p\forall
x\forall a\forall b(x\in A\rightarrow((\exists y(y\in B\wedge
(x,y)\in f))\wedge ((a\in B\wedge b\in B\wedge(x,a)\in f\wedge
(x,b)\in f)\rightarrow a=b))).\]

The first part means that given $x\in V(p)$ and $q\geq p$ if $x\in f(q)$ then $x\in (A\times B)(q)$ therefore $\Vdash_q x=(a,b)$ for some $a\in A(q)$ and $b\in B(q)$. The second part means that given $z,a,b \in V(p)$ and $q\geq p$, if $z\in A(q)$ there exists $y\in V(q)$ such that $y\in B(q)$ and $\Vdash_q (z,y)\in f$; and  that if $a\in B(q)$, $b\in B(q)$, $\Vdash_q (z,a)\in f$ and $\Vdash_q (z,b)\in f$ then $a=b$. We have then:

\begin{teor}\label{function}
Let $A, B\in V(p)$, then $f\in V(p)$ is forced  at $p$as a function from $A$ to $B$ if and only if for every $q\geq p$: 
\[f_q:A(q)\rightarrow B(q)\]
\[f_q(a)=b \text{ where }b\text{ satisfies }\Vdash_q(a,b)\in f\]
is a function.
\end{teor} 

\subsection*{\normalsize{The Axiom of Foundation}}

The axiom of foundation will not be forced in its standard version\footnote{See \cite{villa},\cite{benavides} for an example where the standard version is not forced.} but in an alternative version which is intuitionistically weaker (weak within intuitionistic logic) than the standard one but classically equivalent. The standard version of the Axiom of Foundation (AF) is given by:
\[AF:= \forall x(\exists y(y\in x)\rightarrow\exists y(y\in x\wedge\neg\exists z(z\in x\wedge z\in y))\]
and we will prove:

\begin{teor}
\[\Vdash_p\neg\exists x\neg(\exists y(y\in x)\rightarrow\exists
y(y\in x\wedge \neg\exists z(z\in x\wedge z\in y)))\]
\end{teor}

\begin{proof}
We have to show that for every $q\geq p$, $x\in V(Q)$ there exists $r\geq q$ such that for any $t\geq r$, $x(t)\neq \emptyset$ implies the existence of $y'\in x(t)$ such that for every $s\geq t$, $x(s)\cap y'(s)=\emptyset$. Suppose this is not the case, then there exists $q\geq p$ and $f\in V(q)$ such that for any $r\geq q$ there exists $t\geq r$ such that $f(t)\neq \emptyset$ and for all $ y'\in f(t)$ there exists $s\geq t$ such that $f(s)\cap y'(s)\neq \emptyset$. Let $g\in f(t)$, and consider $s\geq t$ such that there exists $f_1 \in f(s)\cap g(s)$. In the same way since $f_1\in f(s)$ there exists $u\geq t$ such that $f(u)\cap f_1(u)\neq \emptyset$. Let $f_2\in f(u)\cap f_1(u)$ and continue the process. We obtain and infinite chain
\[f_1\in f(u)\] \[f_2\in f_1(u)\] \[f_3\in f_2(u_2)\]\[.\]\[.\]\[.\]\[f_{n+1}\in f_n(u_n)\]\[.\]\[.\]\[.\]
which satisfies
\[ran(f_1)>ran(f_1(u_1))>ran(f_2)>...>ran(f_n(u_n))>f_{n+1}>...\]
but this contradicts the classical axiom of foundation on the external model.
 \end{proof}

\subsection*{\normalsize{The Axiom of Choice}}

As with the axiom of foundation, the axiom of choice will not be forced in all its classical formulations, for example its more natural version which asserts the existence of an election function for any family of non empty sets will not be always forced (see \cite{rive},\cite{benavides}). Instead, it will be forced an intuitionistic equivalent formulation to the G\" odel translation of the next version of the axiom of choice:

\begin{center}
AC:=For any family $\mathcal{F}$ of disjoint non empty sets, there exists a set $E$ such that  $E\cap X=\{z\}$ for all $x\in\mathcal{F}$.
\end{center}

In formal language we have,

\[AC:=\forall \mathcal{F}(\forall x\forall y \varphi\rightarrow \psi),\]
where
\[\varphi:= x\in \mathcal{F}\wedge y\in \mathcal{F}\rightarrow ((\exists z(z\in x\wedge z\in y)\rightarrow x=y)\wedge \exists z(z\in x))\]
\[\psi:= \exists E(x\in \mathcal{F} \rightarrow  \exists a(a\in E\wedge a\in x\wedge \forall b(b\in E\wedge b\in x \rightarrow b=a)))\]
We will force $\neg\neg AC^{*}$, where $AC^{*}$ is the formula derived from replace each subformula of the form $\forall x\phi(x)$ in $AC$ by $\forall x\neg\neg \phi^{*}$; that we know is equivalent to the G\"odel translation of AC. Before proving that $\neg\neg AC^{*}$ is forced we still need a further simplification. From intuitionistic logic we know 

\begin{equation}\label{int1}
\Vdash_p \neg\neg (\alpha\rightarrow \beta) \leftrightarrow \alpha \rightarrow \neg\neg \beta
\end{equation}
\begin{equation}\label{int2}
\Vdash_p \neg\neg \forall x\neg\neg \varphi \leftrightarrow \forall x \neg\neg \varphi
\end{equation}

Then 

\begin{align*}
\Vdash_p \neg\neg AC* \leftrightarrow & \neg\neg \forall \mathcal{F} \neg\neg(\forall x\forall y \varphi\rightarrow \psi)^{*}\\
\leftrightarrow & \neg\neg \forall \mathcal{F} \neg\neg(\forall x \neg\neg(\forall y \varphi \rightarrow \psi)^{*})\\
\leftrightarrow & \neg\neg\forall \mathcal{F} \neg\neg(\forall x \neg\neg(\forall y \neg\neg (\varphi \rightarrow \psi^{*}) \text{ (because $\varphi$ does not contain $\forall$)}\\
\leftrightarrow & \forall \mathcal{F}(\forall x\forall y \varphi \rightarrow \neg\neg \psi^{*}) \text{ (by \ref{int1}, \ref{int2})}
\end{align*}

\begin{teor}
Given $p\in\mathbb{P}$,
\[ \Vdash_p \neg\neg AC^{*}\]
\end{teor}

\begin{proof}
Let $\mathcal{F},X, Y \in V(p)$ and for $q\geq p$ suppose that $\Vdash_q\varphi$, we want to show that $\Vdash_q \neg\neg \psi^{*}$. By \ref{int1} we have
\begin{equation}\label{int3}
\Vdash_q \psi^{*} \leftrightarrow \exists E(X\in\mathcal{F}\rightarrow \exists a(a\in E\wedge a\in X\wedge \forall b(b\in E\wedge b\in X\rightarrow \neg\neg (b=a))).
\end{equation}
Consider a well order on $\bigcup_{q\geq p} \mathcal{F}(q)$, then we can write
\[\bigcup_{q\geq p} \mathcal{F}(q)=\{x_{\alpha}: \alpha <\gamma\}.\]
We will choose elements of $\{X_{\alpha}(q)\}_{\alpha\in\gamma}$ in the next way. Given $\alpha$ if we already choose one element of $X_{\alpha}(q)$ we do not do anything and we pass to $X_{\alpha+1}$, otherwise take $x_{\alpha}$ in $X_{\alpha}(q)$ and take $x_{\beta}=x_{\alpha}\upharpoonright_{[r)}$ in $X_{\beta}=X_{\alpha}\upharpoonright_{[r)}$. Define
\[E:[q)\rightarrow \bigcup_{r\geq q} V(q)\]
\[ E(q)=\{x_{\alpha}:X_{\alpha}\in \mathcal{F}(q)\}\]
We have $E(q)\subset V(q)$ because $X_{\alpha}(q)\subset V(q)$ for any $\alpha$, on the other hand given $r\geq q$, if $x_{\alpha}\in E(q)$ we have $x_{\alpha}\upharpoonright _{[r)}$ is the choose element of $X_{\alpha}(r)$ by construction then $x_{\alpha}\upharpoonright _{[r)}\in E(r)$; thus we can conclude that $E\in V(q)$. Note also that by construction we have that $X_{\alpha}(q)\cap E(q)$ contain just one element.  To show that $\Vdash_q\neg\neg \psi^{*}$ we need to find $t\geq r\geq q$ such that $\Vdash_t \psi^{*}$. Given $X\in \mathcal{F}(q)$ then $X\upharpoonright \in \mathcal{F}(r)$, therefore, $X\upharpoonright_{[r)}=X_{\alpha}$ for some $\alpha\in\gamma$. Thus, there exists $t\geq r$ and $X_{\beta}\in \mathcal{F}(t)$ such that $X_{\beta}=X_{\alpha}\upharpoonright_{[t)}$. In the step $\beta$ we choose an element  $x_{\beta}\in X_{\beta}(t)$, $x_{\beta}\in E(t)$ and $x_{\beta}\in X(t)=X_{\alpha}(t)=X_{\beta}(t)$ and is the unique element which satisfies both conditions, thus we have shown that 
\[\Vdash_t \exists E(X\in\mathcal{F}\rightarrow \exists a(a\in E\wedge a\in X\wedge \forall b(b\in E\wedge b\in X\rightarrow \neg\neg (b=a))),\]
by \ref{int3} we have the result.
\end{proof}

From the above lemmas follow then

\begin{teor}
Given $\mathbb{P}$ a partial order with the topology of the hereditary subsets we have
\[\mathbb{V}\Vdash_{\mathbb{P}} ZFC,\]
or in ther words, for every $p\in\mathbb{P}$, $\Vdash_p ZFC$.
\end{teor}

And from the fundamental theorem of model theory \ref{ftmt} we have:

\begin{teor}
Given $\mathbb{P}$ a partial order with the topology of the hereditary subsets and $\mathcal{F}$ a generic filter over $\mathbb{P}$,
\[\mathbb{V}[\mathcal{F}]\models ZFC.\]
\end{teor}

\section{The independence of the Continuum Hypothesis}

We can finally give an alternative and more simple proof of the independence of the Continuum Hypothesis.  The Continuum Hypothesis in its original version, i.e. as  postulated by Cantor, asserts  that every subset of the real numbers is enumerable or it has the cardinality of the Continuum, $2^{\aleph_0}$. From the Dedekind construction of the real numbers we know that the cardinality of the real numbers is the same as the cardinality of the set of subsets of $\mathbb{N}$. By the Cantorian diagonal  argument we now $|\mathbb{N}|<|P(\mathbb{N})|$, thus another formulation of the Continuum hypothesis would  be to say that it does not exists a set which cardinality lies between $|\mathbb{N}|$ and $|P(\mathbb{N})|$. We will find an order $\mathbb{P}$ and a variable set $B$ such that for every $p\in \mathbb{P}$, 
\[\Vdash_p |N|\leq |B|\leq |P(\mathbb{N})|.\]
Thus for any generic filter over $\mathbb{P}$ the inequality will be still valid in the classical collapsed model. We start introducing an important object which corresponds in this context to the so called subobject classifier.

\subsection{The Subobject Classifier}

We know there is a correspondence between the elements of $P(\mathbb{N})$ and the elements of $2^{\mathbb{N}}$, given by the assignation to each set  $S\subseteq \mathbb{N}$ of its characteristic function 
\[\chi_{_S}:\mathbb{N}\rightarrow 2\] 
\[\chi_{_S}(n)=\begin{cases} 0 & \text{ if } n\notin S\\
  1 & \text{ if } n\in S\end{cases}.\]
In this definition the values 0, 1 are representing the truth values in classical logic of the proposition $n\in S$. The subobject classifier will be an object that play an analogous role to 2 in the definition of the characteristic function, but over our variable sets. To construct this object remember that the truth values in the sheaves of structures are given by the open sets  of the base space (see the end of section \ref{lss}). Taking $\mathbb{P}$ a partial order with the topology of hereditary subsets, the truth values are precisely the hereditary subsets. If we denote by $\mathbb{P}^{+}$ the set of hereditary subsets, the object that will define the variable set of these truth values is given by:

\[\Omega:\mathbb{P}\rightarrow \mathbb{P}^{+}\]
\[\Omega(p)=[p)^{+},\]
 
Where $[p)^{+}$ is the set which elements are the hereditary subsets of $[p)$. To see that $\Omega$ play an analogous role as 2 in the definition of the characteristic function we will show the next result 
\[ \Vdash_p |\widehat{\Omega}(p)(p)^{\mathbb{N}}|=|P(\mathbb{N})|, \text{ for all } p\in \mathbb{P},\]
where in the above proposition $\mathbb{N}$ and $P(\mathbb{N})$ are denoting the sets forced as the natural numbers (the minimum inductive set) and its generalized power set. We start showing which is the set forced as the  natural numbers. 

\begin{lema}[$\widehat{\mathbb{N}}(p)$ is the minimum inductive set]
\[\Vdash_p \forall S(\widehat{\emptyset}(p)\in S\wedge (\forall x(x\in S\rightarrow Suc(x)\in S)))\rightarrow \widehat{\mathbb{N}}(p)\subseteq S\]
\end{lema}

\begin{proof}
We want to show that given $q\geq p$ and $f\in V(q)$ we have:
\[\Vdash_q (\widehat{\emptyset}(p)\in f\wedge (\forall x(x\in f\rightarrow Suc(x)\in f)))\rightarrow \widehat{\mathbb{N}}(p)\subseteq f,\]
that means to show that given $t\geq q$ if $\Vdash_t(\widehat{\emptyset}(t)\in f\wedge (\forall x(x\in f\rightarrow Suc(x)\in f)))$ then $\Vdash_t \widehat{N}(t)\subseteq f$. Suppose $\Vdash_t(\widehat{\emptyset}(t)\in f\wedge (\forall x(x\in f\rightarrow Suc(x)\in f)))$, then given $r\geq t$,  we have that  $f$ satisfies   $\widehat{\emptyset}(r)\in f(r)$ and  for every $x\in V(r)$  if $x\in f(r)$ then $Suc(x)\in f(r)$.\\
Now, we want to show that $\widehat{\mathbb{N}}(r)(r)\subseteq f(r)$, suppose this is not the case. Let $n$  minimum such that $\widehat{n}(r)\notin f(r)$, since $\widehat{\emptyset}(r)\in f(r)$ then $n\neq 0$. Let $m$ such that $m+1=n$, then $\widehat{m}(r)\in f(r)$ therefore $Suc(\widehat{m}(s))\in f(s)$ but $Suc(\widehat{m}(s))=\widehat{m+1}(s)=\widehat{n}(s)$ a contradiction. Let $x\in \widehat{\mathbb{N}}(r)(r)$ then $x\in f(r)$, since $r\geq t$ was arbitrary we have
\[\Vdash_t \forall x(x\in \widehat{\mathbb{N}}(t)\rightarrow x\in f).\]
\end{proof}

\begin{corol}
$\Vdash_p \mathbb{N}=\widehat{\mathbb{N}}(p)$ for all $p\in\mathbb{P}$.
\end{corol}

On the other hand using the power set axiom, the variable set forced as $P(\mathbb{N})$ is given by 

 \[G:[p)\rightarrow
 \bigcup_{q\geq p}V(q)\]
 \[G(q)=\{h:[q)\rightarrow
 \bigcup_{r\geq q}\textsl{P}(\widehat{\mathbb{N}}(p)(r)):(r\geq q\Rightarrow h(r)\subseteq \widehat{\mathbb{N}}(p)(r))\wedge (t\geq
 r\geq q\Rightarrow \forall i\in h(r), i\upharpoonright_{[t)}\in
 h(t))\}.\]
 
Now we can construct a function that is forced as a bijection between $G$ and $\widehat{\Omega(p)}(p)^{\widehat{\mathbb{N}}(p)}$. Given $H$ such that $\Vdash_p H\in G$ define 
\[\chi(H)=\chi_{_H}:[p)\rightarrow
 \bigcup_{q\geq p}V(q)\]
 \[\chi_{_H}(q)=\{(\widehat{n}(q),\widehat{K}(q)):n\in\mathbb{N}\wedge
K=\{r\geq p:\Vdash_r\widehat{n}(r)\in H\}\},\]
and
\[\chi:[p)\rightarrow\bigcup_{q\geq p}V(q)\]
 \[\chi(q)=\{(H,\chi_{_H}\upharpoonright_{[q)}):H\in G(p)\},\]
where the objects denoted as ordered pairs are the objects forced as the respective ordered pairs.

\begin{lema}
Let $p\in\mathbb{P}$ and $H$ be such that $\Vdash_p H\in G$, then
\[\Vdash_p `` \chi_{H}\text{ is a function from } \widehat{\mathbb{N}}(p) \text{ to } \widehat{\Omega(p)}(p)"\] 
\end{lema}

\begin{proof}
To begin we have to show that $\chi_H\in V(p)$, it is clear by definition that $\chi_H(q)\subseteq V(q)$. On the other hand if $g\in \chi_H(Q)$ then $\Vdash_q g=(\widehat{n}(q),\widehat{K}(q))$ for some $n\in\mathbb{N}$ and $K=\{r\geq p: \Vdash_r \widehat{n}(r)\in H\}$, then for $t\geq  q$
\[g\upharpoonright_{[t)}=(\widehat{n}(q),\widehat{K}(q))\upharpoonright_{[t)}=
\widehat{(n,K)}(q)\upharpoonright_{[t)}=\widehat{(n,K)}(t)=
(\widehat{n}(t),\widehat{K}(t))\] and
\[(\widehat{n}(t),\widehat{K}(t))\in \chi_{_H}(t)=\{(\widehat{n}(t),\widehat{K}(t)):n\in\omega\wedge
K=\{r\geq p:\Vdash_r\widehat{n}(r)\in H\}\}.\] 
Now we have to show that $\chi_H$ is forced as a function from $\widehat{\mathbb{N}}(p)$ to $\widehat{\Omega(p)}(p)$. If  $\Vdash_p (\widehat{n}(p),\widehat{K}(p))\in \chi_H$ then given $r\in K$ we have that $\widehat{n}(r)\in H(r)$, if $s\geq r$ then, by the definition of $G$, $\widehat{r}\upharpoonright_{[s)}=\widehat{n}(s)\in H(s)$ therefore  $s\in K$, $K\in [p)^{+}$, and $\Vdash_p \widehat{K}(p)\in \widehat{\Omega(p)}(p)$. From the definition we have that  $f:\widehat{N}(q)\rightarrow \widehat{\Omega(q)}(q)$ such that $f(\widehat{n}(q))=\widehat{K}(q)$ with $K$ as defined above is a function for all $q\geq p$ then by theorem \ref{function}, $\chi_h$ is forced as a function from $\widehat{N}(p)$ to $\widehat{\Omega(p)}(p)$.
\end{proof}

\begin{teor}
\[\Vdash_p ``\chi\text{ is a bijective function from }G\text{ to } \widehat{\Omega(p)}(p)^{\widehat{\mathbb{N}}(P)}"\]   
\end{teor}

\begin{proof}
\begin{enumerate}
\item $\Vdash_p ``\chi\text{ is a function from }G \text{ to }\widehat{\Omega(p)}(p)^{\widehat{\mathbb{N}}(p)}$ : Given $q\geq p$ by the above lemma we have $\Vdash_q \chi \subset G\times \widehat{\Omega(p)}(p)$. On the other hand for $H$ such that $\Vdash_q h\in G$, $\chi_H\upharpoonright_{[q)}$ is unique such that $\Vdash_q (H, \chi_h\upharpoonright_{[q)})\in \chi$ then by \ref{function} we have the result.
\item $\Vdash_p \forall x\forall y(x\in G\wedge y\in G\wedge x\neg y)\rightarrow (\chi_x\neq\chi_y)$: Let $t\geq q\geq p$ and suppose for $x,y\in V(t)$ that $\Vdash_t x\in G\wedge y\in G \wedge x\neg y$ then $x,y\in G(t)$ and for all $r\geq t$ , $\nVdash_r x=y$ or in other words for $r\geq t$ there exist $s\geq r$ and $n\in \mathbb{N}$ such that $\Vdash_s \widehat{n}(s)\in x$ and $\nvdash_s \widehat{n}(s)\in y$. We want to show that for $r\geq t$, $\nVdash_r \chi_x=\chi_y$. Let $s, n$ as above then 
\[\Vdash_r\widehat{s}(r)\in\chi_{_x}(\widehat{n}(r))=\{\widehat{q}(r):q\geq
p\wedge\Vdash_q \widehat{n}(q)\in x\}\] 
and
\[\Vdash_r \widehat{s}(r)\notin \chi_y(\widehat{n}(r)).\]
Then $\nVdash \chi_x(\widehat{n}(r))=\chi_y(\widehat{n}(r))$ and $\nVdash_r \chi_x=\chi_y$.
\item $\Vdash_p \forall x(x\in \widehat{\Omega(p)}(p)^{\widehat{\mathbb{N}}(p)}\rightarrow \exists y(y\in G\wedge \chi_y=x)$: Let $X$ such that $\Vdash_p X\in \widehat{\Omega(p)}(p)^{\widehat{\mathbb{N}}(p)}$ and consider 
\[S_X:[p)\rightarrow \bigcup_{q\geq
p}\textsl{P}(\widehat{\mathbb{N}}(p)(q))\]
\[S_X(q)=\{\widehat{n}(q):\Vdash_q\widehat{q}(q)\in X(\widehat{n}(q))\}. \]

Then $\Vdash_p S_X\in G$ and 

\begin{eqnarray*}
\chi{_{_{S_X}}}(q)&=&\{(\widehat{n}(q),\widehat{K}(q)):n\in\mathbb{N}\wedge
K=\{r\geq p:\widehat{n}(r)\in S_X(r)\}\}\\
&=&\{(\widehat{n}(q),\widehat{K}(q)):n\in\mathbb{N}\wedge K=\{r\geq
p:\widehat{n}(r)\in\{\widehat{m}(r):\Vdash_r\widehat{r}(r)\in X(\widehat{m}(r))\}
 \}\}\\
 &=&\{(\widehat{n}(q),\widehat{K}(q)):n\in\mathbb{N}\wedge
K=\{r\geq
p:\Vdash_r\widehat{r}(r)\in X(\widehat{n}(r))\}\}\\
&=&\{(\widehat{n}(q), X(\widehat{n}(q))):n\in\mathbb{N}\}=X(q).
\end{eqnarray*}
We conclude that   $\Vdash_p \chi_{_{S_X}}=X$.
\end{enumerate}
\end{proof}

\subsection{The independence of CH}

We start showing that for all $p\in \mathbb{P}$
\[\Vdash_p |\widehat{\mathbb{N}}(p)|<|\widehat{P(\mathbb{N})}(p)|<|\widehat{P(P(\mathbb{N}))(p)|.}\]

\begin{teor}
\[\Vdash_p \neg \exists f(f\subseteq (\widehat{\mathbb{N}\times P(\mathbb{N})})(p)\wedge  f \text{ is a function }\wedge f \text{ is onto})\] 
\end{teor}

\begin{proof}
We want to show that for every $q\geq p$,
\[\nVdash_q \exists f(f\subseteq (\widehat{\mathbb{N}\times P(\mathbb{N})})(p)\wedge  f \text{ is a function }\wedge f \text{ is onto}).\]
Suppose there exists $t\geq p$ where the proposition is forced and $f\in V(t)$ such that 
\[\Vdash_t (f\subseteq (\widehat{\mathbb{N}\times P(\mathbb{N})})(p)\wedge  f \text{ is a function }\wedge f \text{ is onto}).\]
Consider \[f_1=\{(a,b): \widehat{(a,b)}(t)\in f(t)\},\]
it is easy to show that $f_1$ defines a surjective function from $\mathbb{N}$ to $P(\mathbb{N})$ generating a contradiction.
\end{proof}

In the same way we obtain,

\begin{corol}$\Vdash_p  \neg \exists f(f\subseteq
\widehat{\textsl{P}(\mathbb{N})}(p) \times
\widehat{\textsl{P}(\textsl{P}(\mathbb{N}))}(p)\wedge f $ is a function
$\wedge f$ is onto$)$.
\end{corol} 

And

\begin{corol}\label{card1} 
\[\Vdash_p
|\widehat{\mathbb{N}}(p)|<|\widehat{\textsl{P}(\mathbb{N})}(p)|<
|\widehat{\textsl{P}(\textsl{P}(\mathbb{N}))}(p)|.\]
\end{corol}

Until now the results presented are being proved for any partial order $\mathbb{P}$ with the topology of its hereditary subsets. From this point we will work over the next order,
\[\langle \mathbb{P}, \leq \rangle=\langle fin(P(P(\mathbb{N}))\times \mathbb{N}\rightarrow 2), \subseteq \rangle, \]
where 
\[fin(P(P(\mathbb{N}))\times \mathbb{N}\rightarrow 2)=\{ f: |f|<|\mathbb{N}|\wedge f\text{ is a function }\wedge dom(f)\subset P(P(\mathbb{N}))\times \mathbb{N} \wedge ran(f)\subset 2\}.\]

Given $p\in\mathbb{P}$ define

\[A:[p)\rightarrow\bigcup_{q\geq p}V(q)\]
\[A(q)=\{(\widehat{H}(q),\widehat{n}(q)):q(H,n)=1\}.\]

We have that $\Vdash_p  A\subseteq \widehat{P(P(\mathbb{N}))}(p)\times \widehat{\mathbb{N}}(p)$ and using $A$ we can define

\[\chi_{_{A}}:[p)\rightarrow\bigcup_{q\geq p}V(q)\]
\[\chi_{_{A}}(q)=\{((\widehat{H}(q),\widehat{n}(q)),\widehat{K}(q)):H\in\textsl{P}(\textsl{P}(\mathbb{N})), n\in\mathbb{N}, K=\{r\geq q:
\Vdash_r(\widehat{H}(r),\widehat{n}(r))\in A\},\]
 that satisfies $\Vdash_p \chi_A \text{ is a function from }\widehat{P(P(\mathbb{N}))}(p)\times \widehat{\mathbb{N}}(p) \text{ to }\widehat{\Omega(p)}(p)$. Using $\chi_A$ we can to construct an injective function $\Phi$ from $\widehat{P(P(\mathbb{N}))}(p)$ to $\widehat{\Omega(p)}(p)^{\widehat{\mathbb{N}}(p)}$, in the next way.
 
 \[\Phi(\widehat{H}(p)):[p)\rightarrow\bigcup_{q\geq
p}V(q)\]
\[\Phi(\widehat{H}(p))(q)=\{(\widehat{n}(q),\chi_{_A}(\widehat{H}(q),\widehat{n}(q))):n\in
\mathbb{N}\}\]
 
 \[\Phi:[p)\rightarrow\bigcup_{q\geq p}V(q)\]
\[\Phi(q)=\{(\widehat{H}(q),\Phi(\widehat{H}(q))):H\in\textsl{P}(\textsl{P}(\omega))\}\]

 \begin{teor} For all $p\in\mathbb{P}$
 \[\Vdash_p \Phi \text{ is an injective function from }\widehat{P(P(\mathbb{N}))}(p)\text{ to } \widehat{\Omega(p)}(p)^{\widehat{\mathbb{N}}(p)}\]
 \end{teor}
 
 \begin{proof}
 From the definition  and \ref{function}  it is clear that $\Phi$ is a function in $V(p)$ from $\widehat{P(P(\mathbb{N}))}(p)$ to $\widehat{\Omega(p)}(p)^{\widehat{\mathbb{N}}(p)}$. To show that $\Phi$ is forced as an one to one function in the universe of variable sets  we want to show 
\[\Vdash_p \forall x\forall y((x\in
\widehat{\textsl{P}(\textsl{P}(\omega))}(p)\wedge
y\in\widehat{\textsl{P}(\textsl{P}(\omega))}(p)\wedge x\neq
y)\rightarrow \Phi(x)\neq \Phi(y)).\]
Consider $t\geq q\geq p$ and  $x,y$ such that $\Vdash_t((x\in
\widehat{\textsl{P}(\textsl{P}(\omega))}(t)\wedge
y\in\widehat{\textsl{P}(\textsl{P}(\omega))}(t)\wedge x\neq y)$ then $x=\widehat{H}(t)$ and $y=\widehat{M}(t)$ for some  $H,M\in P(P(\mathbb{N}))$ and for all $r\geq t$, $\nVdash x=y$, then $H\neq M$. Since $t$ is a finite, there exists $n\in\mathbb{N}$ such that $t(H,N)$ and $t(M,n)$ are not defined. Construct $s\geq t$ such that $s(H,n)=1$ and $s(M,n)=0$. If $v\geq s$ note that $\Vdash_v \widehat{s}(v)\in \Phi (\widehat{H}(v))(\widehat{n}(v))=\chi_A (\widehat{H}(v))(\widehat{n}(v))$ because $r(H,n)=1$. On the other hand, since $s(M,n)=0$ we have $\Vdash_v\widehat{s}(v)\not in \Phi(\widehat{H}(s))(\widehat{n}(v))$. From this we conclude that for every $r\geq t$, $\Vdash_t \Phi(\widehat{H}(t))\neq \Phi(\widehat{M}(t))$, therefore, $\Phi$ is forced as an injective function.
 \end{proof}
 
 From this result and corollary \ref{card1} we obtain,
 
\begin{corol} 
 \[\Vdash_p
|\mathbb{N}|<|\widehat{\textsl{P}(\mathbb{N})}(p)|<|\textsl{P}(\mathbb{N})|
\]
\end{corol}
 
Finally from this corollary and theorem \ref{ftmt} we obtain.

\begin{teor} 
Consider the partial order $\langle \mathbb{P}, \leq \rangle=\langle fin(P(P(\mathbb{N}))\times \mathbb{N}\rightarrow 2),\subseteq \rangle$ with the topology of its hereditary subsets. Let $\mathcal{F}$ a generic filter of $V^{\mathbb{P}}$ then 
\[ V^{\mathbb{P}}[\mathcal{F}]\models ZFC + \neg CH.\] 
\end{teor}

\end{document}